\documentclass{amsproc}
\usepackage{amsfonts}
\usepackage{amsmath}
\usepackage{amssymb}
\usepackage{amstext}
\usepackage{accents}
\usepackage{hyperref}

\theoremstyle{plain}

\newtheorem{corollary}{Corollary}

\newtheorem{definition}{Definition}

\newtheorem{lemma}{Lemma}

\newtheorem{proposition}{Proposition}
\newtheorem{remark}{Remark}

\newtheorem{theorem}{Theorem}
\numberwithin{equation}{section}
\input{tcilatex}

\begin{document}
\title[Steenrod and Dyer-Lashof algebras ]{The $\func{mod}2$ Steenrod and
Dyer-Lashof algebras as quotients of a free algebra}
\author{Nondas E. Kechagias }
\address{ Department of Mathematics, University of Ioannina, \\
Greece, 45500}
\email{nkechag@uoi.gr}
\urladdr{http://users.uoi.gr/nkechag/}
\thanks{}

\begin{abstract}
A non-connected neither of finite type Hopf algebra $\mathcal{F}_{0}$\ is
defined over $%
\mathbb{Z}
/2%
\mathbb{Z}
$\ and its $\hom $\ dual turns out to be a tensor product of polynomial
algebras. Certain quotient Hopf algebras include the Steenrod and
Dyer-Lashof algebras. This setting provides a map between the Steenrod
coalgebra and a direct limit of Dyer-Lashof coalgebras.\ 
\end{abstract}

\maketitle

Certain important algebras in algebraic topology admit the structure of a
Hopf algebra and \ their duals turned out to be polynomial algebras. It is
well known that the homology Hopf algebras $H_{\ast }\left( BU;%
\mathbb{Z}
\right) $ and $H_{\ast }\left( BO;%
\mathbb{Z}
/p%
\mathbb{Z}
\right) $\ enjoy a very special property: that they are self-dual, so that
they are isomorphic to the cohomology Hopf algebras $H^{\ast }\left( BU;%
\mathbb{Z}
\right) $\ and $H^{\ast }\left( BO;%
\mathbb{Z}
/p%
\mathbb{Z}
\right) $. Following the original paper of Milnor and Moore \cite%
{Milnor-Moore} people have studied Hopf algebras, mainly connected and of
finite type. Connected structure arise ubiquitously in algebraic topology.
The homology of non-connected but grouplike homotopy associative H-space
leads to the more general setting of component Hopf algebras. Most of the
structure Theorems for connected Hopf algebras are best derived from the
Poincar\'{e}-Birkoff-Witt Theorem. If one drops any one of the hypotheses,
then one is likely to find counterexamples.\ \ \ \ 

The purpose of this note is to connect two of the most well known Hopf
algebras in algebraic topology. The first is the $\func{mod}2$\ Steenrod
algebra and the second is the $\func{mod}2$ Dyer-Lashof algebra. We
re-derive Milnor's and Madsen's results. Milnor \cite{Milnor 1958} showed
that the $\hom $ dual of the $\func{mod}2$ Steenrod algebra is a polynomial
algebra and Madsen \cite{Madsen} proved that the $\hom $ dual of the $\func{%
mod}2$\ Dyer-Lashof algebra turned out to be a tensor product of polynomial
algebras. The first is connected of finite type and the second is not. The
corresponding proofs were not an easy task and their description as algebras
were given in terms of generators. This phenomenon studied by certain
authors, namely the notion of a Hopf algebra been copolynomial i.e. to have
polynomial dual. It is the purpose of this note to provide a general setting
including both the Steenrod (section 2) and Dyer-Lashof algebras (section
3). This setting also provides an induced map from the Steenrod coalgebra to
a direct limit of Dyer-Lashof coalgebras (corollary \ref{final corollary}).

A lot of research has been done concerning similarities between the Steenrod
(\cite{Steenrod}) and Araki-Kudo (\cite{Araki-Kudo}) Dyer-Lashof (\cite%
{Dyer-Lashof}) algebras. Let us just recall a few of them. First of all May'
paper (\cite{May1}) gave a treatment of co-homology operations in an
algebraic setting which included all the usual cases. In his setting $P^{0}$
(the zero operator) need not be the identity operator. He proved that the
Araki-Kudo Dyer-Lashof and Steenrod operations are special cases of a single
general construction. Bisson and Joyal (\cite{Bisson-Joyal}) defined a $Q$%
-ring equipped with square operations satisfying the Adem-Cartan formulae.
It turned out that the homology of the 0-space of the sphere spectrum is the
free $Q$-ring on one generator and the dual of the Steenrod algebra is also
a $Q$-ring provided the zero operator need not be the identity. Pengelley
and Williams (\cite{Pengelley-Williams}) constructed a bigraded algebra,
called the Kudo-Araki-May algebra, $K$ such that it is isomorphic to $%
\mathcal{R}$ as coalgebras. They also proved that there is an algebra
bijection between $\mathcal{A}_{2}$\ and a certain inverse limit of
subspaces of $K$.\ 

It shall also be mentioned that certain Hopf algebras admit the action of
the opposite of the Steenrod algebra. The action of the Steenrod algebra on
certain duals has been studied extensively because of important applications.

We thank Professor May for bringing to our attention many relevant works. \
\ \ 

Let $\mathcal{F}_{0}$ denote the free graded associative algebra with unit
generated by $\left\{ Q^{i}\;|\;i\geq 0\right\} $\ over $%
\Bbb{F}_2%
$\ with degrees $|Q^{i}|=i$. A typical element of $\mathcal{F}_{0}$ is given
by juxtaposition and is of the form 
\begin{equation*}
Q^{I}=Q^{i_{k}}...Q^{i_{1}}
\end{equation*}%
associated with the sequence $I=\left( i_{k},...,i_{1}\right) $, $I\in 
\mathbb{N}
^{k}$.\ Define the degree, length and excess of $I$ by 
\begin{equation*}
|I|=\tsum\limits_{t=1}^{k}i_{t}\text{, }l\left( I\right) =k\text{ and }%
e(I)=i_{k}-\left( \sum\limits_{1}^{k-1}i_{t}\right) \text{ respectively.}
\end{equation*}%
The empty sequence representing the identity element is of degree zero,
length zero, and of infinite excess. Let $\mathcal{F}_{0}^{\left( i\right) }$
denote the vector subspace of $\mathcal{F}_{0}$ spanned by monomials of
degree $i$. Since $\mathcal{F}_{0}^{\left( i\right) }$ is an infinite
dimensional vector space, $\mathcal{F}_{0}$\ is an object considered
intractable. It turns out to be manageable using the length.

$\mathcal{F}_{0}$ admits the structure of a Hopf algebra with coproduct
defined on generators by 
\begin{equation*}
\psi \left( Q^{i}\right) =\tsum_{t=0}^{i}Q^{i-t}\otimes Q^{t}
\end{equation*}%
and with augmentation defined on $\mathcal{F}_{0}^{\left( 0\right) }=%
\Bbb{F}_2%
[Q^{0}]$ by $\varepsilon \left( Q^{0}\cdots Q^{0}\right) =1$. Here $\mathcal{%
F}_{0}^{\left( 0\right) }$ is a polynomial algebra. Thus $\mathcal{F}_{0}$
is not connected, neither of finite type. We prove that the dual of $%
\mathcal{F}_{0}$, $\mathcal{F}_{0}^{\ast }$,\ is a tensor product of
polynomial algebras (Theorem \ref{Basic}). This is a key result. Certain
interesting quotients of $\mathcal{F}_{0}$\ including both the Steenrod and
Dyer-Lashof algebras can be studied in this setting.

Let $R[k]$ denote the subcoalgebra of the Dyer-Lashof Hopf algebra $R$\
consisting of fixed length $k$\ elements (see \ref{length k}). A coalgebra
map $R[k]\rightarrow R[k+1]$ is defined respecting the opposite of the
Steenrod algebra action and the family of these maps induce a direct limit $%
\underrightarrow{\lim }R[k]$. Finally a coalgebra map between the Steenrod
algebra and the previous limit is induced in corollary \ref{final corollary}.

We extend our results to the $p$\ odd prime number case in \cite{Kech3}. But
instead of the Steenrod and Dyer-Lashof algebras one will consider the
sub-Hopf algebras generated by the reduced powers $P^{i}$\ and $Q^{i}$\
respectively.

\section{The free algebra}

For each monomial $Q^{J}$\ we denote its sequence $J$\ by $I\left(
Q^{J}\right) $.\ The sum $I+J$\ and difference $I-J$\ of two sequences of
length $k$\ is defined termwise. $I-J$\ is undefined , if an entry is less
than zero. Then $e\left( I+J\right) =e\left( I\right) +e\left( J\right) $, $%
d\left( I+J\right) =d\left( I\right) +d\left( J\right) $.

\begin{definition}
We define a total ordering on the set of finite sequences as follows. 
\newline
a) $I<J$, if $l\left( I\right) <l\left( J\right) $. \newline
b) If $I=\left( i_{k},\cdots ,i_{1}\right) $ let $I_{t}=\left( i_{t},\cdots
,i_{1}\right) $ for $1\leq t\leq k$. For sequences of the same length we
define $I<J$, if $e\left( I_{t}\right) <e\left( J_{t}\right) $\ for the
smallest $t$\ such that $e\left( I_{t}\right) \neq e\left( J_{t}\right) $.
Observe that $e\left( I_{t}\right) =e\left( J_{t}\right) $ for all $t$\
implies that $I=J$ provided $l\left( I\right) =l\left( J\right) $. \ 
\end{definition}

\begin{lemma}
$\mathcal{F}_{0}$ is a component coalgebra with respect to length and if $%
\pi \mathcal{F}_{0}=\left\{ Q^{I}\;|\;\psi \left( Q^{I}\right) =Q^{I}\otimes
Q^{I}\right\} $, then $\pi \mathcal{F}_{0}$\ is the free monoid generated by 
$Q^{0}$. Moreover, the component $\mathcal{F}_{0}\left[ k\right] $ of $%
\left( Q^{0}\right) ^{k}$\ is the subcoalgebra of $\mathcal{F}_{0}$\ spanned
by 
\begin{equation*}
\left\{ Q^{I}\;|\;l\left( I\right) =k\right\} \text{.}
\end{equation*}%
Thus 
\begin{equation*}
\mathcal{F}_{0}=\tbigoplus\limits_{k\geq 0}\mathcal{F}_{0}\left[ k\right] 
\text{ with }\mathcal{F}_{0}\left[ 0\right] =%
\Bbb{F}_2%
\text{. }
\end{equation*}%
Moreover, $\mathcal{F}_{0}\left[ k\right] \otimes \mathcal{F}_{0}\left[ l%
\right] \rightarrow \mathcal{F}_{0}\left[ k+l\right] $. \ \bigskip
\end{lemma}

We shall determine the primitive elements of $\mathcal{F}_{0}\left[ k\right] 
$. Define%
\begin{eqnarray*}
x_{k,0} &=&\underset{k}{\underbrace{Q^{0}\cdots Q^{0}}}\text{ for }1\leq k%
\text{;} \\
x_{k,i} &=&\underset{k-i}{\underbrace{Q^{0}\cdots Q^{0}}}\underset{i}{%
\underbrace{Q^{1}Q^{0}\cdots Q^{0}}}\text{ for }1\leq i\leq k\text{.}
\end{eqnarray*}%
\ \ \ \bigskip

\begin{lemma}
$x_{k,i}$ is a primitive monomial for $k\geq 1$ and $k\geq i\geq 1$.
Moreover, $\left\{ x_{k,i}\;|\;1\leq i\leq k\right\} $\ is a basis for the
primitive elements of\ $\mathcal{F}_{0}\left[ k\right] $.
\end{lemma}

\begin{proof}
Proceed by induction on $k$. The case $k=1$\ is trivial. Let $I=\left(
i_{k},J\right) $\ with $J=\left( i_{k-1},\cdots ,i_{1}\right) $. There are
two cases to be considered: $i_{k}=0$ and $i_{k}=1$. If $i_{k}=0$, then $%
Q^{J}$\ is a primitive element. Let $i_{k}=1$\ and $J\neq \left( 0,\cdots
,0\right) $. Then $\psi \left( Q^{I}\right) $\ contains the summand $%
Q^{1}Q^{\left( 0,\cdots 0\right) }\otimes Q^{0}Q^{J}$. The claim follows. \ 
\end{proof}

To describe the $\hom $ dual of $\mathcal{F}_{0}$\ we proceed by analogy
with Madsen's computations of the dual of the Dyer-Lashof algebras. The
computation of $\mathcal{F}_{0}\left[ k\right] ^{\ast }$\ is based on
multiplication of the duals of monomials.

Let $I\left[ k\right] $ be the basis of monomials of $\mathcal{F}_{0}\left[ k%
\right] $.

\begin{lemma}
Define $\phi :%
\mathbb{N}
^{k}\rightarrow I\left[ k\right] $\ by $\phi \left( n_{k},\cdots
n_{1}\right) =Q^{\tsum\limits_{1}^{k}n_{t}I\left( x_{k,t}\right) }$. Then $%
\phi $\ is an isomorphism of sets.
\end{lemma}

Since $\mathcal{F}_{0}$\ is a coassociative, cocommutative coalgebra, $%
\mathcal{F}_{0}{}^{\ast }=\tprod \mathcal{F}_{0}\left[ k\right] ^{\ast }$\
is an algebra. We define elements of $\mathcal{F}_{0}\left[ k\right] ^{\ast
} $\ by 
\begin{eqnarray*}
y_{k,o} &=&\left( \left( Q^{0}\right) ^{k}\right) ^{\ast }\text{, }k\geq 0%
\text{;} \\
y_{k,i} &=&\left( x_{k,i}\right) ^{\ast }\text{, }1\leq i\leq k\text{.}
\end{eqnarray*}%
Now $y_{k,0}$\ is the identity element in $\mathcal{F}_{0}\left[ k\right]
^{\ast }$\ and $\tprod y_{k,0}$\ is the identity element in $\mathcal{F}%
_{0}{}^{\ast }$. The augmentation of $\mathcal{F}_{0}{}^{\ast }$\ is given
by 
\begin{equation*}
\varepsilon \tprod a_{k}y_{k,0}=a_{0}\text{. }
\end{equation*}%
$\mathcal{F}_{0}{}^{\ast }$\ is not a coalgebra and $\mathcal{F}_{0}\left[ k%
\right] ^{\ast }\cdot \mathcal{F}_{0}\left[ l\right] ^{\ast }=0$, for $k\neq
l$.\ 

\begin{theorem}
\label{Basic}$\mathcal{F}_{0}\left[ k\right] ^{\ast }$ is a polynomial
algebra on $\left\{ y_{k,i}\;|\;1\leq i\leq k\right\} $.
\end{theorem}

\begin{proof}
Let $\left\langle -,-\right\rangle :\mathcal{F}_{0}\left[ k\right] ^{\ast
}\otimes \mathcal{F}_{0}\left[ k\right] \rightarrow 
\Bbb{F}_2%
$\ be the Kronecker product and $\Phi _{k}:\mathcal{F}_{0}\left[ k\right]
\rightarrow \mathcal{F}_{0}\left[ k+1\right] $ be the map given by $\Phi
_{k}\left( Q^{I}\right) =Q^{0}Q^{I}$. $\Phi $\ is a degree preserving map of
coalgebras such that 
\begin{equation*}
\Phi \left( x_{k,i}\right) =x_{k+1,i+1}\ \text{and dually }\Phi ^{\ast
}\left( y_{k+1,i+1}\right) =y_{k,i}\ \text{for }i\leq k\text{ and }\Phi
^{\ast }\left( y_{k+1,1}\right) =0\text{.}
\end{equation*}%
We compare the polynomial algebra on $\left\{ y_{k,i}\;|\;1\leq i\leq
k\right\} $\ with $\mathcal{F}_{0}\left[ k\right] ^{\ast }$.\ 

Let $\Lambda =\left( \lambda _{k},\cdots ,\lambda _{1}\right) \in 
\mathbb{N}
^{k}$ and 
\begin{equation*}
\xi ^{\Lambda }=\tprod y_{k,i}^{\lambda _{i}}\text{.}
\end{equation*}%
Let $Q^{\Lambda \left( I\right) }=Q^{\lambda _{k}}\cdots Q^{\lambda _{1}}$\
and $\tsum \lambda _{i}=\lambda $. Here $\Lambda \left( I\right) =\tsum
\lambda _{i}I\left( x_{k,i}\right) $. We claim that 
\begin{equation*}
\left\langle \xi ^{\Lambda },Q^{J}\right\rangle \equiv 1\func{mod}2\text{
implies }J\geq \Lambda \left( I\right) \text{.}
\end{equation*}%
Let $\psi \left( \lambda \right) :\mathcal{F}_{0}\left[ k\right] \rightarrow
\left( \mathcal{F}_{0}\left[ k\right] \right) ^{\lambda }$ be the iterated
coproduct. Then 
\begin{equation*}
\psi \left( \lambda \right) \left( Q^{J}\right) =\tsum Q^{J_{i}}\otimes
\cdots \otimes Q^{J_{\lambda }}\text{ such that }\tsum J_{t}=J\text{.}
\end{equation*}%
Let $J=\left( \mu _{k},\cdots ,\mu _{1}\right) $\ with $\tsum \mu
_{i}=\lambda $. 
\begin{equation*}
\left\langle \xi ^{\Lambda },Q^{J}\right\rangle \equiv 1\func{mod}2\text{
implies }\mu _{1}\geq \lambda _{1}\text{ and \ }J\geq \Lambda \left(
I\right) \text{.}
\end{equation*}%
Next we consider the case $\mu _{1}=\lambda _{1}$. Let $\Lambda ^{\prime
}=\left( \lambda _{k},\cdots ,\lambda _{2},0\right) $ and $J^{\prime
}=J-\lambda _{1}I\left( x_{k,1}\right) $. Then $\left\langle \xi ^{\Lambda
^{\prime }},Q^{J\prime }\right\rangle \equiv 1\func{mod}2$ and 
\begin{equation*}
Q^{J^{\prime }}=\Phi \left( Q^{J"}\right) \text{ such that }\tsum_{2}^{k}\mu
_{t}I_{k-1,t-1}=J^{"}\text{. }
\end{equation*}%
Let $\Lambda ^{"}=\left( \lambda _{k},\cdots ,\lambda _{2}\right) $, then 
\begin{eqnarray*}
\left\langle \xi ^{\Lambda ^{\prime }},\Phi \left( Q^{J"}\right)
\right\rangle &\equiv &1\func{mod}2\text{ implies }\left\langle \Phi ^{\ast
}\left( \xi ^{\Lambda ^{\prime }}\right) ,Q^{J"}\right\rangle \equiv 1\func{%
mod}2\text{ and } \\
\left\langle \xi ^{\Lambda ^{"}},Q^{J"}\right\rangle &\equiv &1\func{mod}2%
\text{.}
\end{eqnarray*}%
As before $\mu _{2}\geq \lambda _{2}$. Continuing in this fashion $J\geq
\Lambda \left( I\right) $.

Now, $%
\Bbb{F}_2%
\left[ y_{k,i}\;|\;1\leq i\leq k\right] $ is contained in $\mathcal{F}_{0}%
\left[ k\right] ^{\ast }$, but the two vector spaces have the same dimension
in each degree. \ \bigskip
\end{proof}

\begin{corollary}
$\mathcal{F}_{0}{}^{\ast }=\underleftarrow{\lim }\otimes \mathcal{F}_{0}%
\left[ k\right] ^{\ast }=\tprod \mathcal{F}_{0}\left[ k\right] ^{\ast }$
which is not a Hopf algebra\ but it admits a coproduct on its generators:%
\begin{equation*}
\psi y_{k,i}=\tsum\limits_{t=1}^{i-1}y_{k-t,0}\otimes
y_{t,i}+\tsum\limits_{t=i}^{k-1}y_{k-t,i-t}\otimes y_{t,0}\text{.}
\end{equation*}
\end{corollary}

A map 
\begin{equation}
\varphi _{k}:\mathcal{F}_{0}\left[ k\right] \rightarrow \mathcal{F}_{0}\left[
k+1\right] \text{ is defined by }\varphi _{k}Q^{I}=Q^{I}Q^{0}\text{.}
\label{direct F_0}
\end{equation}%
It is a coalgebra injection map\ and the direct limit, $\underrightarrow{%
\lim }\mathcal{F}_{0}\left[ k\right] $, is defined under these maps.\ \ 

\begin{definition}
\label{Steenrod action}Let $\mathcal{A}$\ denote the $\func{mod}2$\ Steenrod
algebra. We define a left action of the Steenrod algebra $\mathcal{A}$\ on $%
\mathcal{F}_{0}$\ as follows%
\begin{eqnarray*}
Sq^{a}Q^{b} &=&\binom{b-a}{a}Q^{b-a}\text{ for }b\geq 0\text{;} \\
Sq^{a}\left( Q^{0}r\right) &=&Q^{0}Sq^{a}\left( r\right) \text{;} \\
Sq^{a}\left( Q^{b}r\right) &=&\tsum\limits_{t}\binom{b-a}{a-2t}%
Q^{b-a+t}Sq^{t}\left( r\right) \text{ for }b>0\text{.}
\end{eqnarray*}%
Here $Q^{j}=0$, if $0>j$.
\end{definition}

Compare with the Nishida relations \cite{Nishida}. The previous action of $%
\mathcal{A}$\ on $\mathcal{F}_{0}$ is called the action of the opposite of
the Steenrod algebra and is abbreviated by $\mathcal{A}^{op}$.\ \ \ 

The next lemma follows.

\begin{lemma}
$\mathcal{F}_{0}\left[ k\right] $\ is closed under the $\mathcal{A}^{op}$\
action and $\mathcal{F}_{0}$\ is an unstable coalgebra over $\mathcal{A}%
^{op} $.\ \ 
\end{lemma}

On the $\hom $ dual $\mathcal{F}_{0}\left[ k\right] ^{\ast }$ we have.

\begin{lemma}
$Sq^{a}y_{k,i}^{b}=\binom{b}{a}y_{k,i}^{b+a}$.
\end{lemma}

\begin{proof}
It follows from the action of $\mathcal{A}^{op}$ on $x_{k,i}$\ and the fact
that $y_{k,i}=\left( x_{k,i}\right) ^{\ast }$. \ \bigskip
\end{proof}

We wish to examine two relatively different directions: one obtained by the
relation $Q^{0}=1$ and the second by eliminating elements of negative
excess. 
\begin{equation*}
\begin{array}{ccccc}
&  & \mathcal{F}_{0} &  &  \\ 
& \swarrow &  & \searrow &  \\ 
\mathcal{F}:= & \mathcal{F}_{0}/\left\langle Q^{0}-1\right\rangle &  & U:= & 
\mathcal{F}_{0}/\left\langle \text{negative excess}\right\rangle%
\end{array}%
\end{equation*}%
The first provides the Steenrod algebra $\mathcal{A}_{2}$ and the second the
Dyer-Lashof algebra $\mathcal{R}$. \ 

\section{The direction to the Steenrod algebra}

Let $\mathcal{F}$ be the quotient of $\mathcal{F}_{0}$\ by the ideal
generated by $Q^{0}-1$. In other words $\mathcal{F}$\ is the free
associative algebra with unit generated by $\left\{ Q^{i}\;|\;i\geq
1\right\} $\ over $%
\Bbb{F}_2%
$\ with degrees $|Q^{i}|=i$.\ $\mathcal{F}$\ is a quotient Hopf algebra of $%
\mathcal{F}_{0}$\ of finite type.

\begin{definition}
Let $\mathcal{F}\left( k\right) $ be the vector subspace generated by all
monomials $Q^{S}$\ such that $l\left( S\right) \leq k$.
\end{definition}

Now $\mathcal{F}\left( k\right) \subseteq \mathcal{F}\left( k+1\right) $\
and $\mathcal{F}=\underrightarrow{\lim }\mathcal{F}\left( k\right) $.\ \ 

\begin{lemma}
$\mathcal{F}\left( k\right) $\ is a coalgebra closed under the $\mathcal{A}%
^{op}$\ action and $\mathcal{F}$\ is an unstable coalgebra over $\mathcal{A}%
^{op}$.\ \ 
\end{lemma}

We note that $\mathcal{F}$ is isomorphic with the $\func{mod}2$ homology of
the space $\Omega \Sigma CP^{\infty }$\ and it has been studied in relation
with quasi-symmetric functions (\cite{Malvenuto}).

There exists an obvious coalgebra epimorphism \ \ \ \ 
\begin{equation*}
\mathcal{F}_{0}\left[ k\right] \twoheadrightarrow \mathcal{F}\left( k\right) 
\text{.}
\end{equation*}%
The next proposition follows.

\begin{proposition}
There exists an injection of $\mathcal{A}$-maps induced by the quotient map
above:%
\begin{equation*}
\mathcal{F}\left( k\right) ^{\ast }\rightarrowtail \mathcal{F}\left[ k\right]
^{\ast }\text{.}
\end{equation*}
\end{proposition}

As a corollary to our result we get a result proved by Crossley.

\begin{corollary}
(\cite{Crossley1999}) $\mathcal{F}^{\ast }$ is a polynomial algebra.
\end{corollary}

The proof of Crossley uses the Borel-Hopf theorem. It follows that $\mathcal{%
F}$ is a copolynomial algebra. Moreover, Crossley following work of Gelfand
et all (\cite{Gelfand}) computed the number of generators in each degree of
the algebra $\mathcal{F}^{\ast }$.

Let $\mathcal{A}_{2}$ be the quotient of $\mathcal{F}$ by the ideal $I\left(
A\right) $ generated by the elements known as the \textbf{Adem relations }in
cohomology 
\begin{equation*}
Q^{i}Q^{j}-\sum\limits_{0}^{[i/2]}\binom{j-k-1}{i-2k}Q^{i+j-k}Q^{k}\text{, }i%
\text{ and }j>0\text{ such that }i<2j\text{.}
\end{equation*}%
We recall that a monomial $Q^{S}=Q^{s_{k}}...Q^{s_{1}}$ is called admissible
in $\mathcal{A}_{2}$, if $s_{k}\geq 2s_{k-1}$, $\cdots $, $s_{2}\geq 2s_{1}$%
. \ \ 

It is obvious that $\mathcal{A}_{2}$\ is isomorphic with the Steenrod
algebra as Hopf algebras \cite{Steenrod}.\ \ \ 

We admit that it is unorthodox to use the symbols $Q^{i}$ for the Steenrod
algebra, but we would like to distinguishing between the action of the
opposite of the Steenrod algebra on these algebras.\ 

\begin{remark}
$\mathcal{A}_{2}=\mathcal{F}/I\left( A\right) $\ is not closed under the $%
\mathcal{A}^{op}$\ action.
\end{remark}

We examine the action of $Sq_{\ast }^{2}$ on $Q^{3}Q^{2}$\ before and after
applying the Adem relations.

\begin{eqnarray*}
Sq_{\ast }^{2}Q^{3}Q^{2} &=&\binom{3-2}{2-2\cdot 1}Q^{3-2+1}Sq_{\ast
}^{1}Q^{2}=Q^{2}Q^{1}\text{but} \\
Q^{3}Q^{2} &=&0\text{, because of Adem relations.}
\end{eqnarray*}

According to the Adem relations, $Q^{1}Q^{2}=Q^{3}Q^{0}=Q^{3}$. Hence after
applying the iterated coproduct on a monomial $Q^{I}$\ of length $k$, it may
be the case that a summand turns out to be of length less than $k$ after
applying the Adem relations. So we define $\mathcal{A}_{2}\left( k\right) $
to be the vector space spanned by admissible monomials of length at most $k$%
. Thus $\mathcal{A}_{2}\left( k\right) \hookrightarrow \mathcal{A}_{2}\left(
k+1\right) $ and $\mathcal{A}_{2}=\underrightarrow{\lim }\mathcal{A}%
_{2}\left( k\right) $. Moreover, we have the following obvious coalgebra
quotient maps:\ 
\begin{equation}
\mathcal{F}_{0}\left[ k\right] \twoheadrightarrow \mathcal{F}\left( k\right)
\twoheadrightarrow \mathcal{A}_{2}\left( k\right) \text{ and }%
\underrightarrow{\lim }\mathcal{F}_{0}\left[ k\right] \twoheadrightarrow 
\mathcal{F}\twoheadrightarrow \mathcal{A}_{2}\text{.}
\label{direct F_0 F A_2}
\end{equation}%
\ \ \ 

Following Milnor we describe the primitive monomials of positive degree for
the convenience of the reader. \ \ \ 

\begin{theorem}
a) A Milnor basis for primitive elements of $\mathcal{A}_{2}$\ is given by%
\begin{equation*}
\left\{ Q^{S_{n+1,n+1}}=Q^{2^{n}}\cdots Q^{2}Q^{1}\;|\;0\leq n\right\} \text{%
.}
\end{equation*}

b) Let $M\mathcal{A}_{2}\left( k\right) $ be the subset of $\mathcal{A}%
_{2}\left( k\right) $\ consisting of all admissible monomials on $\mathcal{A}%
_{2}\left( k\right) $. Then the map $f:%
\mathbb{N}
^{k}\rightarrow M\mathcal{A}_{2}\left( k\right) $ given by 
\begin{equation*}
f\left( n_{1},\cdots ,n_{k}\right) =Q^{\Sigma n_{i}S_{i,i}}
\end{equation*}%
is an isomorphism of sets. Moreover, $f^{-1}:M\mathcal{A}_{2}\left( k\right)
\rightarrow 
\mathbb{N}
^{k}$ is given by 
\begin{equation*}
f^{-1}\left( Q^{s_{k}}...Q^{s_{1}}\right) =\left( n_{1},\cdots ,n_{k}\right) 
\text{ such that }n_{k-t}=s_{t+1}-2s_{t}-\cdots -2^{t}s_{1}\text{.\ }
\end{equation*}%
Here $0\leq t\leq k-1$.

c) If $Q^{S}$\ is not admissible in $\mathcal{A}_{2}\left( k\right) $ and $%
Q^{S}=\tsum a_{J}Q^{J}$\ after applying the Adem relations, then $%
a_{J}\equiv 1\func{mod}2$\ implies $J<S$.

\begin{proof}
This is a bookkeeping application. \ \ 
\end{proof}
\end{theorem}

As a corollary to theorem \ref{Basic} we obtain Milnor's result that the
Steenrod algebra is copolynomial.

\begin{theorem}
\cite{Milnor 1958} $\mathcal{A}^{\ast }$\ is a polynomial algebra and a set
of generators is given by%
\begin{equation*}
\left\{ \xi _{n}:=\left( Q^{S_{n+1,n+1}}\right) ^{\ast }\;|\;0\leq n\right\} 
\text{.}
\end{equation*}
\end{theorem}

A special case of a copolynomial algebra is a bipolynomial algebra which is
a graded connected bicommutative Hopf algebra such that both it and its dual
are polynomial algebras. Husemoller (\cite{Husemoller}) studied a universal
construction of this kind of algebras and Ravenel and Wilson (\cite%
{Ravenel-Wilson}) proved that being bipolynomial of finite type over $%
\Bbb{F}_2%
$ determines the Hopf algebra structure.

For completeness we quote Ravenel and Wilson 's result.

\begin{theorem}
(\cite{Ravenel-Wilson}) $\mathcal{P}$ is bipolynomial isomorphic to its dual
as Hopf algebras.\ \ 
\end{theorem}

All the above have applications on the $\Omega $ spectrum for Brown-Peterson
cohomology.

Next, let $\mathcal{P}$\ denote the polynomial algebra $%
\Bbb{F}_2%
\left[ Q^{i}\;|\;1\leq i\right] $ given as a quotient of $\mathcal{F}$\ by
the commutative relation. As a corollary to our main Theorem we obtain that $%
\mathcal{P}$ is a bipolynomial Hopf algebra.

A basis for the primitive elements for $\mathcal{P}$\ using Newton's
polynomials and an idea described by May (\cite{C-L-M})\ can be found in
Wellington's book (\cite{Wellington}). We note that Madsen (\cite{Madsen})
used that set to provide a basis for the primitive elements of $H_{\ast
}\left( Q_{0}S^{0}\right) $.

We also note that $\mathcal{P}$\ is isomorphic to a Hopf subalgebra of $%
H_{\ast }\left( Q_{0}S^{0}\right) $ and the action of $\mathcal{A}^{op}$\ on
its primitive elements is known (\cite{Wellington}).

\section{The direction to the Dyer-Lashof algebra\ }

Next we proceed to the other direction, namely eliminating elements of
negative excess.

\begin{definition}
Let $\mathcal{U}$\ be the quotient of $\mathcal{F}_{0}$\ by the ideal
generated by elements of negative excess.
\end{definition}

$\mathcal{U}$ is a quotient Hopf algebra of $\mathcal{F}_{0}$\ not connected
neither of finite type. As before $\mathcal{U}\left[ k\right] $\ stands for
the subspace generated by monomials of fixed length $k$.

We recall that each admissible sequence admits a unique decomposition with
respect to the primitive monomial elements, therefore we repeat the
appropriate results from \cite{Kech2} for the convenience of the reader.

\begin{definition}
Let $t$\ be a natural number such that $1\leq t\leq k$. \newline
Define $J_{k,t}=\left( \underset{t}{\underbrace{2^{t-1},...,2,1,1}},\underset%
{k-t}{\underbrace{0,...,0}}\right) $. The degree of $J_{k,t}$\ is $2^{t}$.\ 
\end{definition}

\begin{proposition}
\cite{Kech2}Let $Q^{J}$ be an\ admissible monomial of length $k$\ and of
positive degree. Then there exist natural numbers $\left( a_{1},a_{1},\cdots
,a_{k}\right) $\ such that $J$\ can be written uniquely as\ \ \ \ 
\begin{equation*}
J=\tsum\limits_{t=1}^{k}a_{t}J_{k,t}\text{. }
\end{equation*}
\end{proposition}

\begin{lemma}
$\mathcal{U}\left[ k\right] $\ is closed under the $\mathcal{A}^{op}$\
action and $\mathcal{U}$\ is an unstable coalgebra over $\mathcal{A}^{op}$.\
\ 
\end{lemma}

\begin{proof}
The excess condition must be checked. Let $b<c$, then $Sq_{\ast
}^{a}Q^{b}Q^{c}=\tsum \binom{b-a}{a-2t}\binom{c-t}{t}Q^{b-a+t}Q^{c-t}$.

Restrictions should be taken into account. $2t\leq a$ implies $b+2t<c+a$, so
we get $b+2t<c+a$. Finally, $b-a+t<c-t$.\ \ 
\end{proof}

The opposite Steenrod coalgebra epimorphism 
\begin{equation*}
\mathcal{F}_{0}\left[ k\right] \twoheadrightarrow \mathcal{U}\left[ k\right]
\end{equation*}%
induces the following corollary.

\begin{corollary}
$\mathcal{U}\left[ k\right] ^{\ast }$ is a polynomial algebra closed under
the Steenrod algebra action.
\end{corollary}

As in the diagram (\ref{direct F_0 F A_2}), the quotient maps above induce
the map 
\begin{equation*}
\underrightarrow{\lim }\mathcal{F}_{0}\left[ k\right] \twoheadrightarrow 
\underrightarrow{\lim }\mathcal{U}\left[ k\right] \text{.}
\end{equation*}%
\ 

The Dyer-Lashof algebra $\mathcal{R}$ is given as the quotient of $\mathcal{U%
}$ modulo the ideal generated by the Adem relations in homology:%
\begin{equation*}
Q^{r}Q^{s}-\sum\limits_{t>0}\binom{t-s-1}{2t-r}Q^{r+s-t}Q^{t}\text{ such
that }r>2s\text{.}
\end{equation*}

\begin{definition}
\label{length k}Let $k\geq 1$, $R[k]$\ is the sub $\mathcal{A}$-coalgebra
spanned by admissible monomials of fixed length $k$ and non-negative excess. 
\begin{equation*}
\{Q^{(I,\varepsilon )}\;|\;(I,\varepsilon )\text{ admissible, }l\left(
I\right) =k\text{ and }e\left( I,\varepsilon \right) \geq 0\text{ }\}\text{. 
}
\end{equation*}
\end{definition}

It is a non-negatively graded Hopf algebra and also a component coalgebra $%
R=\bigoplus\limits_{k\geq 0}R[k]$ with respect to the length.

We recall that each admissible sequence admits a unique decomposition with
respect to the primitive monomial elements. We repeat the appropriate
results from \cite{Madsen}.

\begin{definition}
Let $t$\ be a natural number such that $0\leq t\leq k-1$ and the term $%
2^{-1} $ is omitted\ in the following expressions. \newline
Define $I_{k,t}=\left( 2^{k-1}-2^{t-1},...,2^{k-t}-1,2^{k-t-1},...,1\right) $%
. The degree of $I_{k,t}$\ is $2^{k}-2^{t}$.\ 
\end{definition}

\begin{proposition}
\cite{Madsen}\label{madsen primitive decomposition}Let $Q^{I}$ be an\
admissible monomial of length $k$\ and of positive degree. Then there exist
natural numbers $\left( a_{0},a_{1},\cdots ,a_{k-1}\right) $\ such that $I$\
can be written uniquely as\ \ \ \ 
\begin{equation*}
I=\tsum\limits_{t=0}^{k-1}a_{t}I_{k,t}\text{. }
\end{equation*}
\end{proposition}

It is known that $R\left[ k\right] $\ respects the action of the opposite of
the Steenrod algebra.\ 

\begin{lemma}
\cite{Madsen}The Dyer-Lashof algebra $\mathcal{R}$\ is an unstable coalgebra
over $\mathcal{A}^{op}$.\ Moreover, $\mathcal{R}\left[ k\right] $\ is closed
under the $\mathcal{A}^{op}$\ action.
\end{lemma}

\begin{proof}
The proof uses the fact that $H_{\ast }(Q_{0}S^{0})$ is an unstable
coalgebra over $\mathcal{A}^{op}$ and the evaluation map 
\begin{equation*}
e:\mathcal{R}\rightarrow H_{\ast }(QS^{0})\text{ given by }e\left(
Q^{I}\right) =Q^{I}i_{0}\text{.}
\end{equation*}%
Here $i_{0}$ is the fundamental class of $H_{0}(S^{0})\hookrightarrow
H_{\ast }(QS^{0})$. \ 
\end{proof}

The opposite Steenrod coalgebra epimorphisms 
\begin{equation*}
\mathcal{F}_{0}\left[ k\right] \twoheadrightarrow \mathcal{U}\left[ k\right]
\twoheadrightarrow \mathcal{R}\left[ k\right]
\end{equation*}%
induce the following corollary.

\begin{corollary}
$\mathcal{R}\left[ k\right] ^{\ast }$ is a polynomial algebra closed under
the Steenrod algebra action.
\end{corollary}

Let us examine the induced map $\varphi _{k}:\mathcal{R}\left[ k\right]
\rightarrow \mathcal{R}\left[ k+1\right] $ given by $\varphi
_{k}Q^{I}=Q^{I}Q^{0}$ (see \ref{direct F_0}). It is not an injection any
longer because of the Adem relations, neither an onto map. For example $%
\varphi _{1}Q^{1}=Q^{1}Q^{0}=0$ and there is no $Q^{i}$\ such that $\varphi
_{1}Q^{i}=Q^{2}Q^{1}$.

\begin{proposition}
The map $\varphi _{k}:\mathcal{R}\left[ k\right] \rightarrow \mathcal{R}%
\left[ k+1\right] $ given by $\varphi _{k}Q^{I}=Q^{I}Q^{0}$ is an $\mathcal{A%
}^{opp}$-coalgebra map for $k\geq 1$. Moreover, \newline
i) if $I\notin 2%
\mathbb{N}
^{k}$\ and $I$\ is admissible, then $\varphi _{k}Q^{I}=0$;\newline
ii) Let $I=\tsum\limits_{i=0}^{k-1}2a_{i}I_{k,i}$, then $\varphi
_{k}Q^{I}=Q^{J}$ where $J=\tsum\limits_{i=0}^{k-1}a_{i}I_{k+1,i+1}$.
\end{proposition}

\begin{proof}
We start with length 2 sequences in order to demonstrate our method. Let $%
I=\left( 2a_{0}+a_{1},a_{0}+a_{1}\right) $.\ \ There are three cases to be
considered:\newline
i) $a_{0}+a_{1}\equiv 1\func{mod}2$; ii) $a_{0}+a_{1}\equiv 0\func{mod}2$\
and $a_{0}\equiv 1\func{mod}2$; and \newline
iii) $a_{0}$, $a_{1}\equiv 0\func{mod}2$.

i) $Q^{a_{0}+a_{1}}Q^{0}=\tsum \binom{t-1}{2t-\left( a_{0}+a_{1}\right) }%
Q^{a_{0}+a_{1}-t}Q^{t}$. We must also consider the restrictions:\newline
$a_{0}+a_{1}-t\geq t$\ and $2t-\left( a_{0}+a_{1}\right) \geq 0$. Therefor $%
2t=a_{0}+a_{1}$, which contradicts our assumption.\ Hence, 
\begin{equation*}
Q^{I}Q^{0}=0\text{.}
\end{equation*}%
\ 

ii) $Q^{a_{0}+a_{1}}Q^{0}=\tsum \binom{t-1}{2t-\left( a_{0}+a_{1}\right) }%
Q^{a_{0}+a_{1}-t}Q^{t}$. By the argument above, we have:%
\begin{equation*}
Q^{a_{0}+a_{1}}Q^{0}=Q^{\frac{a_{0}+a_{1}}{2}}Q^{\frac{a_{0}+a_{1}}{2}}\text{%
.}
\end{equation*}%
Now we consider the monomial $Q^{2a_{0}+a_{1}}Q^{\frac{a_{0}+a_{1}}{2}}Q^{%
\frac{a_{0}+a_{1}}{2}}$ and in particular the term $Q^{2a_{0}+a_{1}}Q^{\frac{%
a_{0}+a_{1}}{2}}$. 
\begin{equation*}
Q^{2a_{0}+a_{1}}Q^{\frac{a_{0}+a_{1}}{2}}Q^{\frac{a_{0}+a_{1}}{2}}=\tsum 
\binom{t-1-\frac{a_{0}+a_{1}}{2}}{2t-\left( 2a_{0}+a_{1}\right) }%
Q^{2a_{0}+a_{1}+\frac{a_{0}+a_{1}}{2}-t}Q^{t}Q^{\frac{a_{0}+a_{1}}{2}}\text{.%
}
\end{equation*}%
We shall also take into consideration the restrictions:%
\begin{equation*}
2t-\left( 2a_{0}+a_{1}\right) \geq 0\text{ and }2a_{0}+a_{1}+\frac{%
a_{0}+a_{1}}{2}-t\geq t+\frac{a_{0}+a_{1}}{2}\text{.}
\end{equation*}%
Therefor $2t=2a_{0}+a_{1}$\ which contradicts our assumption.\ \ 

iii) In this case we consider the monomial 
\begin{equation*}
Q^{2a_{0}+a_{1}}Q^{\frac{a_{0}+a_{1}}{2}}Q^{\frac{a_{0}+a_{1}}{2}}\text{.}
\end{equation*}%
Applying the\ Adem relations we get 
\begin{equation*}
Q^{2a_{0}+a_{1}}Q^{\frac{a_{0}+a_{1}}{2}}Q^{\frac{a_{0}+a_{1}}{2}}=Q^{\frac{%
2a_{0}+a_{1}}{2}+\frac{a_{0}+a_{1}}{2}}Q^{\frac{2a_{0}+a_{1}}{2}}Q^{\frac{%
a_{0}+a_{1}}{2}}=
\end{equation*}%
\begin{equation*}
Q^{\frac{a_{0}}{2}+a_{0}+a_{1}}Q^{\frac{a_{0}}{2}+\frac{a_{0}+a_{1}}{2}}Q^{%
\frac{a_{0}+a_{1}}{2}}\text{.}
\end{equation*}

For the general case let $I=\tsum\limits_{i=0}^{k-1}a_{i}I_{k,i}=\left(
i_{k},\cdots ,i_{1}\right) $, we call 
\begin{equation*}
i_{t}=b_{k-t}\text{ such that }b_{j}=2b_{j+1}-a_{j+1}\text{ for }j\leq k-2%
\text{\ and }b_{k-1}=\tsum a_{i}\text{.}
\end{equation*}%
\ \ 

Let $\left( a_{0},\cdots ,a_{k-1}\right) \notin 2%
\mathbb{N}
^{k}$\ and there exists an $a_{k-t_{0}}\equiv 1\func{mod}2$\ with $t_{0}$\
the smallest among the appropriate indices. Therefore $a_{k-t_{0}+s}\equiv 0%
\func{mod}2$\ for $s>0$.

If $b_{k-1}=i_{1}\equiv 1\func{mod}2$, then $Q^{i_{1}}Q^{0}=0$\ by the Adem
relations and the restrictions.

If $b_{k-1}=i_{1}\equiv 0\func{mod}2$, then $Q^{i_{1}}Q^{0}=Q^{\frac{i_{1}}{2%
}}Q^{\frac{i_{1}}{2}}$.

Let $b_{k-1}=i_{1}\equiv 0\func{mod}2$ and $a_{k-t_{0}}\equiv 1\func{mod}2$\
as above. In this case, we consider the following monomial:%
\begin{equation*}
Q^{b_{0}}\cdots Q^{b_{k-t_{0}-1}}Q^{\frac{b_{k-t_{0}}+\cdots +b_{k-1}}{2}}Q^{%
\frac{b_{k-t_{0}}}{2}}\cdots Q^{\frac{b_{k-1}}{2}}\text{.}
\end{equation*}%
Now $b_{k-t_{0}-1}\equiv 1\func{mod}2$\ and the Adem relations imply 
\begin{equation*}
Q^{b_{k-t_{0}-1}}Q^{\frac{b_{k-t_{0}}+\cdots +b_{k-1}}{2}}Q^{\frac{%
b_{k-t_{0}}}{2}}=
\end{equation*}%
\begin{equation*}
\tsum \binom{t-\frac{b_{k-t_{0}}+\cdots +b_{k-1}}{2}-1}{2t-b_{k-t_{0}-1}}%
Q^{b_{k-t_{0}-1}+\frac{b_{k-t_{0}}+\cdots +b_{k-1}}{2}-t}Q^{t}Q^{\frac{%
b_{k-t_{0}}}{2}}\text{.}
\end{equation*}%
Because of the restrictions $t=\frac{b_{k-t_{0}-1}}{2}$ which contradicts
our assumption.

If we had $b_{k-t_{0}-1}\equiv 0\func{mod}2$, then we would have $t=\frac{%
b_{k-t_{0}-1}}{2}$\ and 
\begin{equation*}
Q^{b_{k-t_{0}-1}+\frac{b_{k-t_{0}}+\cdots +b_{k-1}}{2}-t}Q^{t}Q^{\frac{%
b_{k-t_{0}}}{2}}=Q^{\frac{b_{k-t_{0}-1}+\cdots +b_{k-1}}{2}}Q^{\frac{%
b_{k-t_{0}-1}}{2}}Q^{\frac{b_{k-t_{0}}}{2}}\text{.}
\end{equation*}

We proceed to the Steenrod algebra action by considering the following
diagram. 
\begin{equation*}
\begin{array}{ccccc}
& Q^{J^{\prime }} & \overset{\varphi _{k}}{\longmapsto } & Q^{J} &  \\ 
Sq_{\ast }^{2^{m}} & \downarrow &  & \downarrow & Sq_{\ast }^{2^{m}} \\ 
& Q^{2I_{k,i}} & \overset{\varphi _{k}}{\longmapsto } & Q^{I_{k+1,i+1}} & 
\end{array}%
\end{equation*}%
Madsen proved that $Sq_{\ast }^{2^{m}}Q^{J}=Q^{I_{k+1,i+1}}$\ if and only if 
$J=I_{k+1,i}$ for $m=i<k$ or $J=I_{k+1,i+1}+I_{k+1,k}$ for $i+1\leq m=k$ 
\cite{Madsen}. We prove the analogue statement for $Q^{2I_{k,i}}$.\ 

\textit{Claim}: $Sq_{\ast }^{2^{m}}Q^{J^{\prime }}=Q^{2I_{k,i}}$\ if and
only if 
\begin{equation*}
J^{\prime }=\left\{ 
\begin{array}{c}
2I_{k,i-1}\text{ for }m=i<k\text{;} \\ 
2I_{k,i}+2I_{k,k-1}\text{ for }i+1<m=k\text{;} \\ 
2I_{k,k-1}\text{ for }i+1=m=k\text{.}%
\end{array}%
\right.
\end{equation*}

\textit{Proof of claim}. Here $m\leq k$.\ First we consider what expression $%
J^{\prime }$ admits\ by degree arguments.%
\begin{equation*}
|J^{^{\prime
}}|=2^{k+1}-2^{i+1}+2^{m}=\tsum\limits_{0}^{k-1}a_{t}|I_{k,t}|=\tsum%
\limits_{0}^{k-1}a_{t}2^{k}-2^{t}\text{.}
\end{equation*}%
There are three cases to consider. In all cases induction is applied.\newline
a) $m=i$\ and $J^{\prime }=2I_{k,i-1}$.\newline
b) $m=k$, $i<k-1$\ and $J^{\prime }=2I_{k,i}+2I_{k,k-1}$.\ \newline
c) $m=k-1$\ and $J^{\prime }=2I_{k,i}+I_{k,k-1}$.

a) We recall that 
\begin{equation*}
2I_{k,i}=\left( 2^{k}-2^{i},\cdots ,2^{k+1-i}-2,2^{k-i},\cdots ,2\right) 
\text{ and }
\end{equation*}%
\begin{equation*}
I_{k,k-1}=\left( 2^{k-2},\cdots ,,2,1,1\right) \text{.}
\end{equation*}%
From the action of the opposite of the Steenrod algebra (\ref{Steenrod
action}) we get 
\begin{equation*}
Sq_{\ast }^{2^{i}}Q^{2^{k}-2^{i-1}}=\tsum \binom{2^{k}-2^{i-1}-2^{i}}{%
2^{i}-2t}Q^{2^{k}-2^{i-1}-2^{i}+t}Sq_{\ast }^{t}\text{.}
\end{equation*}%
For $t=2^{i-1}$\ we get $Q^{2^{k}-2^{i}}Sq_{\ast }^{2^{i-1}}$. If $t<2^{i-1}$%
, then $2^{i}-2^{t}$\ contains a summand $2^{s}$ such that $s<i-1$ and $%
2^{k}-2^{i-1}-2^{i}$\ does not. Therefor $\binom{2^{k}-2^{i-1}-2^{i}}{%
2^{i}-2t}\equiv 0\func{mod}2$. Inductively we get \ 
\begin{equation*}
Sq_{\ast }^{2^{i}}Q^{2I_{k,i-1}}=Q^{\left( 2^{k}-2^{i},\cdots
,2^{k+1-i}-2\right) }Sq_{\ast }^{2}Q^{\left( 2^{k+1-i},\cdots ,2\right) }%
\text{.}
\end{equation*}%
Now $\binom{2^{k+1-i}-2}{2-2t}\equiv 1$ for $t=0$\ or $1$. For $t=0$\ we get
the expected element%
\begin{equation*}
Sq_{\ast }^{2^{i}}Q^{2I_{k,i-1}}=Q^{2I_{k,i}}\text{.}
\end{equation*}%
For $t=1$\ we get \ 
\begin{equation*}
Q^{\left( 2^{k}-2^{i},\cdots ,2^{k+1-i}-2\right) }Q^{2^{k-i+1}-1}Sq_{\ast
}^{1}Q^{2^{k-i}}Q^{\left( 2^{k-1-i},\cdots ,2\right) }=
\end{equation*}%
\begin{equation*}
Q^{\left( 2^{k}-2^{i},\cdots ,2^{k+1-i}-2\right)
}Q^{2^{k-i+1}-1}Q^{2^{k-i}-1}Q^{\left( 2^{k-1-i},\cdots ,2\right) }
\end{equation*}%
and the last monomial has excess zero.

The case b) is similar. We proceed to case c).

c) $2I_{k,i}+I_{k,k-1}=$\newline
$\left( 2^{k}-2^{i}+2^{k-2},\cdots
,2^{k+1-i}-2+2^{k-i-1},2^{k-i}+2^{k-i-2},\cdots ,2^{2}+1,2+1\right) $. 
\begin{equation*}
Sq_{\ast }^{2^{k-1}}Q^{2^{k}-2^{i}+2^{k-2}}Q^{2I_{k-1,i-1}+I_{k-1,k-2}}=
\end{equation*}%
\begin{equation*}
\tsum \binom{2^{k-1}-2^{i}+2^{k-2}}{2^{k-1}-2t}%
Q^{2^{k-1}-2^{i}+2^{k-2}+t}Sq_{\ast }^{t}Q^{2I_{k-1,i-1}+I_{k-1,k-2}}\text{.}
\end{equation*}%
The case $t<2^{k-2}$\ is eliminated because of negative excess:%
\begin{equation*}
2^{k-1}-2^{i}+2^{k-2}+2t<2^{k}-2^{i}+2^{k-2}\text{.}
\end{equation*}%
Therefore only the case $t=2^{k-2}$\ remains and this case ends up at 
\begin{equation*}
Sq_{\ast }^{2}Q^{2^{2}+1}Q^{2+1}=Q^{2+1}Q^{2+1}+Q^{2^{2}}Sq_{\ast
}^{1}Q^{2+1}=Q^{2+1}Q^{2+1}\text{.}
\end{equation*}%
Finally, $Sq_{\ast
}^{2^{k-1}}Q^{2I_{k,i}+I_{k,k-1}}=Q^{2I_{k,k-1}+I_{k,i+1}}\neq Q^{2I_{k,i}}$.
\end{proof}

The following corollary is just an application of the last proposition.

\begin{corollary}
Let $I=\tsum\limits_{i=0}^{k-1}a_{i}I_{k,i}$ such that $a_{i}\equiv 0\func{%
mod}2^{n}$ for all $i$,\ then 
\begin{equation*}
\varphi _{k+n-1}\cdots \varphi _{k}\left( Q^{I}\underset{n}{\underbrace{%
Q^{0}\cdots Q^{0}}}\right) =Q^{J}
\end{equation*}%
where $J=\tsum\limits_{i=0}^{k-1}b_{i}I_{k+n,i+n}$ and $a_{i}=2^{n}b_{i}$\
for all $i$.\ 
\end{corollary}

We end this section by summarizing the results above in the following two
diagrams.

\begin{equation*}
\begin{array}{ccccccccc}
&  &  &  &  &  &  &  &  \\ 
& \mathcal{F}_{0}\left[ k\right] &  &  &  &  & \underset{\rightarrowtail }{%
\lim }\mathcal{F}_{0}\left[ k\right] &  &  \\ 
\swarrow &  & \searrow &  &  & \swarrow &  & \searrow &  \\ 
\mathcal{F}_{0}\left( k\right) &  & \mathcal{U}\left[ k\right] &  & \underset%
{\rightarrowtail }{\lim }\mathcal{F}\left( k\right) &  &  &  & \underset{%
\rightarrowtail }{\lim }\mathcal{U}\left[ k\right] \\ 
\downarrow &  &  &  & \downarrow &  &  &  &  \\ 
\mathcal{A}_{2}\left( k\right) &  & \downarrow &  & \underset{%
\rightarrowtail }{\lim }\mathcal{A}_{2}\left( k\right) &  &  &  & \downarrow
\\ 
&  & \mathcal{R}\left[ k\right] &  &  &  &  &  & \underset{\rightarrow }{%
\lim }\mathcal{R}\left[ k\right]%
\end{array}%
\end{equation*}

\section{A map between the Steenrod and Dyer-Lashof coalgebras}

We conclude this note by filling the diagram above discussing the induced
map between $\mathcal{A}_{2}\left( k\right) $\ and $\mathcal{R}\left[ k%
\right] $.

\begin{equation*}
\begin{array}{ccc}
&  & \mathcal{U}\left[ k\right] \\ 
& \swarrow &  \\ 
\mathcal{A}_{2}\left( k\right) &  & \downarrow \\ 
& \searrow &  \\ 
&  & \mathcal{R}\left[ k\right]%
\end{array}%
\end{equation*}%
\bigskip

Let $\phi _{k}:\mathcal{U}\left[ k\right] \rightarrow \mathcal{A}_{2}\left(
k\right) $ be the obvious map given by 
\begin{equation*}
\phi _{k}\left( Q^{I}Q^{0}\cdots Q^{0}\right) =Q^{I}\text{.}
\end{equation*}
This is not an injection map. For example $\phi _{2}\left( Q^{1}Q^{1}\right)
=0$. But it is a coalgebra epimorphism:\newline
Let $Q^{I}\in \mathcal{A}_{2}\left( k\right) $\ with $l\left( I\right)
=m\leq k$\ and $I$\ an admissible sequence. Then, if $I=\left( i_{m},\cdots
,i_{1}\right) $, $i_{t}\geq 2i_{t-1}$\ for all $t$, $\phi _{k}\left(
Q^{I}Q^{0}\cdots Q^{0}\right) =Q^{I}$. \ 

We extend the map $\phi _{k}$ from $\mathcal{A}_{2}\left( k\right) $\ to $%
\mathcal{R}\left[ k\right] $\ in the obvious way:%
\begin{equation*}
\pi _{k}:\mathcal{A}_{2}\left( k\right) \rightarrow \mathcal{R}\left[ k%
\right] \text{ given by }\pi _{k}\left( Q^{I}\right) =\left( Q^{I}\underset{n%
}{\underbrace{Q^{0}\cdots Q^{0}}}\right) \text{\ for }n=k-l\left( I\right) 
\text{\ \ }
\end{equation*}%
\begin{equation*}
\text{and }\pi _{k}\left( Q^{I}\right) =Q^{I}\text{ for }k=l\left( I\right) 
\text{.}
\end{equation*}%
We must note that the last diagram is not a commutative diagram under the
given maps. For example: $Q^{1}Q^{1}\in \mathcal{U}\left[ k\right] \cap 
\mathcal{R}\left[ k\right] $\ but 
\begin{equation*}
\pi _{2}\phi _{2}\left( Q^{1}Q^{1}\right) =0\text{.}
\end{equation*}%
Nevertheless, the map $\pi _{k}$\ is an epimorphism of coalgebras.\ \ \ 

\begin{theorem}
The map $\pi _{k}:\mathcal{A}_{2}\left( k\right) \rightarrow \mathcal{R}%
\left[ k\right] $\ is an epimorphism of coalgebras.
\end{theorem}

\begin{proof}
Let $Q^{I}\in \mathcal{R}^{^{\prime }}\left[ k\right] $ such that $%
I=\tsum\limits_{t=0}^{k-1}a_{t}I_{k,t}$\ and $\tsum%
\limits_{t=0}^{k-1}a_{t}>0 $\ (\ref{madsen primitive decomposition}). We
shall define an element $K\in \mathcal{A}_{2}\left( k\right) $ such that $%
\pi _{k}\left( K\right) =Q^{I}$.

Let $Q^{J\left( I\right) }\in \mathcal{A}_{2}\left( k\right) $ such that 
\begin{equation*}
J\left( I\right) =\tsum\limits_{t=0}^{k-1}2^{t}a_{t}I_{k-t,0}\text{.}
\end{equation*}%
Namely, if $J\left( I\right) =\left( b_{k},\cdots ,b_{1}\right) $, then $%
b_{t+1}=2^{t}\left( a_{0}+\cdots +a_{t}\right) $\ for $0\leq t\leq k-1$. Let
us also define $J_{t}\left( I\right) =\left( b_{k},\cdots ,b_{t+1}\right) $\
for $0\leq t\leq k-1$.\ 

Let us first consider the case $I=\left( 2a_{0}+a_{1},a_{0}+a_{1}\right) $.
In this case $J\left( I\right) =\left( 2\left( a_{0}+a_{1}\right)
,a_{0}\right) $. We start with $Q^{2\left( a_{0}+a_{1}\right) }Q^{a_{0}}$\
and apply the Adem relations in the Dyer-Lashof algebra.%
\begin{equation*}
Q^{2\left( a_{0}+a_{1}\right) }Q^{a_{0}}\overset{Adem}{=}\tsum \binom{%
t-a_{0}-1}{2t-2\left( a_{0}+a_{1}\right) }Q^{\left( 2+1\right)
a_{0}+2a_{1}-t}Q^{t}\text{.}
\end{equation*}%
Here we shall take into account the restrictions: 
\begin{equation*}
\left( 2+1\right) a_{0}+2a_{1}\geq 2t\text{.}
\end{equation*}%
For $t=a_{0}+a_{1}$, we have the required term 
\begin{equation}
Q^{2a_{0}+a_{1}}Q^{a_{0}+a_{1}}\text{.}
\end{equation}%
For $t=a_{0}+a_{1}+s$ such that $2s\leq a_{0}$\ and for $\binom{s+a_{1}-1}{2s%
}\neq 0\func{mod}2$, we have the non-required term 
\begin{equation*}
Q^{2a_{0}+a_{1}-s}Q^{a_{0}+a_{1}+s}\text{.}
\end{equation*}%
In other words the monomial $Q^{J_{2}}Q^{2a_{0}+a_{1}-s}Q^{a_{0}+a_{1}+s}$
appears and it will be eliminated by adding an appropriate term $%
Q^{J^{\left( 1\right) }}$\ to $Q^{J\left( I\right) }$, where 
\begin{equation*}
J^{\left( 1\right) }=(a_{0}+a_{1}+s,a_{0}-s)\text{.}
\end{equation*}%
This is because the term $Q^{a_{0}+a_{1}+s}Q^{a_{0}-2s}$ will provide $%
Q^{2a_{0}+a_{1}-s}Q^{a_{0}+a_{1}+s}$\ after applying the Adem relations. We
start with the smallest $s$ such that 
\begin{equation*}
\binom{s+a_{1}-1}{2s}\neq 0\func{mod}2
\end{equation*}%
\ and apply the Adem relations. Neither of the terms $%
Q^{a_{0}+a_{1}+s}Q^{a_{0}-2s}$ will provide $Q^{2a_{0}+a_{1}}Q^{a_{0}+a_{1}}$%
.\ \ 

It is also important to note that $J^{\left( 1\right) }<J\left( I\right) $\
and this property will be present in our procedure. In each step we define
and use a new sequence less than the previous ones so that the non-required
term is eliminated.

We describe our method by starting with $J\left( I\right) $.

Let us first recall the coefficient involved in the Adem relations $\binom{%
t-s-1}{2t-r}$ for $Q^{r}Q^{s}$ such that $r>2s$. Second, $b_{t+1}\equiv 0%
\func{mod}2^{t}$\ for $t>0$. We start with the sequence $J\left( I\right) $
and apply the relations between the first and the second element from the
left and proceed to the pair consisting of the new just defined element and
the next one. We continue in this fashion up to the pair where there is no
relation. Each time a relation is applied in the sequence $J^{\left(
s\right) }$ a new sequence is defined abbreviated by $J^{\left( s+1\right) }$
starting with $J\left( I\right) =J^{\left( 0\right) }$. Each time a relation
is applied between $b_{i}^{\left( s\right) }$\ and $b_{i-1}^{\left( s\right)
}$\ we first consider only the case $t=\frac{b_{i}^{\left( s\right) }}{2}$\
and call%
\begin{equation*}
b_{i}^{\left( s+1\right) }=b_{i}^{\left( s\right) }+b_{i-1}^{\left( s\right)
}-\frac{b_{i}^{\left( s\right) }}{2}\text{ and }b_{i-1}^{\left( s+1\right) }=%
\frac{b_{i}^{\left( s\right) }}{2}\text{.}
\end{equation*}%
Therefore a new sequence is defined $J^{\left( s+1\right) }=$ 
\begin{equation*}
\left( b_{k}^{\left( s+1\right) }=b_{k}^{\left( s\right) },\cdots
,b_{i+1}^{\left( s+1\right) }=b_{i+1}^{\left( s\right) },b_{i}^{\left(
s+1\right) },b_{i-1}^{\left( s+1\right) },b_{i-2}^{\left( s+1\right)
}=b_{i-2}^{\left( s\right) },\cdots ,b_{1}^{\left( s+1\right)
}=b_{1}^{\left( s\right) }\right) .
\end{equation*}%
Following the pattern above we get%
\begin{equation*}
J^{\left( k-1\right) }=\left( b_{k}^{\left( k-1\right) },\cdots
,b_{1}^{\left( k-1\right) }\right)
\end{equation*}%
with $b_{1}^{\left( k-1\right) }=\tsum\limits_{0}^{k-1}a_{t}$\ and $%
b_{i}^{\left( k-1\right)
}=2^{i-1}\tsum\limits_{0}^{k-1}a_{t}-2^{i-2}\tsum\limits_{0}^{i-1}a_{t}$ for 
$k\geq i>1$.

According to our procedure 
\begin{equation*}
b_{1}^{\left( k-1-k-2\right) }=\tsum\limits_{0}^{k-1}a_{t}\text{, }%
b_{2}^{\left( k-1-k-2\right) }=2\tsum\limits_{0}^{k-1}a_{t}-a_{k-1}\text{
and }
\end{equation*}%
\begin{equation*}
b_{i}^{\left( k-1+k-2\right)
}=2^{i-1}\tsum\limits_{0}^{k-1}a_{t}-2^{i-2}\tsum%
\limits_{0}^{i-1}a_{t}-2^{i-3}\tsum\limits_{0}^{i-2}a_{t}\text{ for }k\geq
i>2\text{.}
\end{equation*}%
We note that there is no relation between $b_{2}^{\left( k-1+k-2\right) }$
and $b_{1}^{\left( k-1+k-2\right) }=b_{1}^{\left( k-1\right) }$. After $%
\frac{k\left( k-1\right) }{2}$ steps we have $J^{\left( \frac{k\left(
k-1\right) }{2}\right) }=I$. To finish off we also have to consider terms
coming from the Adem relations with corresponding $t$\ such that $t>\frac{a}{%
2}$. Our method is described after we demonstrate an example on the method
above.

Example for $k=3$.\newline
$I=a_{0}I_{3,0}+a_{1}I_{3,1}+a_{2}I_{3,2}$ and $J=\left(
2^{2}\tsum\limits_{0}^{2}a_{t},2\tsum\limits_{0}^{1}a_{t},a_{0}\right) $.%
\newline
$J^{\left( 1\right) }=\left(
2^{2}\tsum\limits_{0}^{2}a_{t}-2a_{2},2\tsum\limits_{0}^{2}a_{t},a_{0}%
\right) $, $J^{\left( 2\right) }=\left(
2^{2}\tsum\limits_{0}^{2}a_{t}-2a_{2},2\tsum\limits_{0}^{2}a_{t}-\tsum%
\limits_{1}^{2}a_{t},\tsum\limits_{0}^{2}a_{t}\right) $\newline
and $J^{\left( 3\right) }=\left(
2^{2}\tsum\limits_{0}^{2}a_{t}-2a_{2}-a_{1},2\tsum\limits_{0}^{2}a_{t}-\tsum%
\limits_{1}^{2}a_{t},\tsum\limits_{0}^{2}a_{t}\right) $.\ 

Given $Q^{I}\in \mathcal{R}^{^{\prime }}\left[ k\right] $ such that $%
I=\tsum\limits_{t=0}^{k-1}a_{t}I_{k,t}$\ and $\tsum%
\limits_{t=0}^{k-1}a_{t}>0 $\ a sequence $J\left( I\right) $ defined above
such that $Q^{J\left( I\right) }\in \mathcal{A}_{2}\left( k\right) $ and $%
J\left( I\right) =\tsum\limits_{t=0}^{k-1}2^{t}a_{t}I_{k-t,0}$. The Adem
relations are applied on $c_{I}Q^{J\left( I\right) }$: 
\begin{equation*}
Q^{J\left( I\right) }\overset{Adem}{=}Q^{I}+\tsum c_{t}\left( I\right)
Q^{I_{t}\left( I\right) }\text{.}
\end{equation*}%
Here $Q^{I_{t}\left( I\right) }$\ is an admissible\ monomial and $%
I_{t}\left( I\right) >I$\ because of the Adem relations. Next we consider $%
t_{0}$ such that $I_{t_{0}}\left( I\right) <I_{t}\left( I\right) $\ for $%
t\neq t_{0}$. A new sequence $J\left( I_{t_{0}}\left( I\right) \right) $\ is
defined as above such that $Q^{J\left( I_{t_{0}}\left( I\right) \right) }\in 
\mathcal{A}_{2}\left( k\right) $ and 
\begin{equation*}
Q^{J\left( I_{t_{0}}\left( I\right) \right) }\overset{Adem}{=}%
Q^{I_{t_{0}}\left( I\right) }+\tsum c_{s}\left( I_{t_{0}}\right)
Q^{I_{s}\left( I_{t_{0}}\right) }\text{.}
\end{equation*}%
Here again $Q^{I_{s}\left( I_{t_{0}}\right) }$\ is an admissible\ monomial
and $I_{s}\left( I_{t_{0}}\right) >I_{t_{0}}\left( I\right) $. Moreover, $%
J\left( I_{t_{0}}\left( I\right) \right) $\ $<J\left( I\right) $. Now we
take 
\begin{equation*}
Q^{J\left( I\right) }+Q^{J\left( I_{t_{0}}\left( I\right) \right) }\overset{%
Adem}{=}Q^{I}+\tsum b_{r}Q^{I_{r}\left( I,I_{t_{0}}\right) }
\end{equation*}%
and consider $r_{0}$ such that $I_{r_{0}}\left( I,I_{t_{0}}\right)
<I_{r}\left( I,I_{t_{0}}\right) $\ for $r\neq r_{0}$.

Proceeding in this fashion an element $K\in \mathcal{A}_{2}\left( k\right) $
will be defined in finite steps such that $\pi _{k}\left( K\right) =Q^{I}$.
\ 
\end{proof}

\begin{corollary}
\label{final corollary}There exists an induced coalgebra map $\pi :\mathcal{A%
}_{2}\rightarrow \underset{\rightarrow }{\lim }\mathcal{R}\left[ k\right] $.
\end{corollary}

\end{document}